\documentclass[reqno]{amsart}
\usepackage{amsbsy,amssymb,amscd,amsfonts,latexsym,amstext,delarray,
amsmath,graphicx}
\newenvironment{nitemize}
{\begin{list}{$\bullet$}
{\setlength{\parsep}{1ex}
\setlength{\topsep}{1.1ex}
\setlength{\partopsep}{0ex}
\setlength{\labelwidth}{0.5cm}
\setlength{\itemindent}{0cm}
\setlength{\itemsep}{0cm}
\setlength{\leftmargin}{0.7cm}
\setlength{\labelsep}{0.2cm}}}
{\end{list}}

\newtheorem{theorem}{Theorem}[section]

\newtheorem{corollary}[theorem]{Corollary}
\newtheorem{lemma}[theorem]{Lemma}
\newtheorem{property}[theorem]{Property}

\theoremstyle{remark}
\newtheorem{remark}[theorem]{Remark}

\theoremstyle{definition}

\theoremstyle{example}

\numberwithin{equation}{section}

\def\Z{\mathbb Z}

\def\cE{\mathcal{E}}

\def\cM{\mathcal{M}}

\newcommand{\ie}{{\it i.e.\/}\ }

\newcommand{\cf}{{\it cf.\/}\ }

\def\text{\hbox}

\renewcommand{\limsup}{\ensuremath{\overline{\lim}\ }}
\renewcommand{\liminf}{\ensuremath{\underline{\lim}\ }}

\title[]{Strengthened large deviations for rational maps and full shifts, with unified proof}

\author{Henri Comman}
\address{Department of Mathematics, University
of Santiago de Chile, Bernardo O'Higgins 3363, Santiago, Chile}
\email{hcomman@mat.usach.cl}

\date{}
\begin{document}


\begin{abstract}
For any hyperbolic rational map and any net of Borel probability measures
on the space of Borel probability measures on the  Julia set, we show that this net  satisfies a strong form of the
large deviation principle with a  rate function given by the
entropy map if and only if the large deviation and the pressure
functionals coincide. To each such principles
corresponds an expression for the entropy of invariant measures.
We give the explicit form of the rate function of  the corresponding  large deviation
principle in the real line for the net of image measures obtained by evaluating the function
$\log|T'|$.  These results  are applied to various examples
including those considered in the literature where only upper bounds
have been proved.
The proof rests on some entropy-approximation property (independent of the net of measures),  which in a suitable formulation, is nothing but the hypothesis involving exposed points in Baldi's theorem. In particular, it
  works verbatim for general dynamical systems. After  stating  the corresponding  general version, as another example we consider the  multidimensional full shift for which the above property has been recently proved, and   we establish  large deviation principles for nets of measures analogous to those of the rational maps case.
  \end{abstract}

\maketitle


\section{Introduction}\label{intro}

The aim of this paper is threefold; it presents large deviation principles arising from  two specific  dynamical systems, which turn to be typical examples of a general  scheme valid for any dynamical system. Let us  begin by detailing   the first example.

\subsection{The case of rational maps}

Let  $T$ be  a hyperbolic rational map of degree at least  $2$ on the
Riemann sphere, and let $J$ be  its Julia set.
Large deviations for
   various types of sequences  of  probability measures
  on the set $\mathcal{M}(J)$ of Borel probability measures on
 $J$ have been studied in the literature.
 We refer here  to the
distribution of pre-images and periodic points
 (\cite{pol_CMP_96}, \cite{Pollicott-Sridharan-07-JDSGT}), and  the
 Birkhoff averages with respect to the measure of
maximal entropy (\cite{lop}). However, except for this last case,
only the large deviation upper-bounds have been proved, and in all
cases the rate function has the form
\begin{equation}\label{rate-function-intro}
I^f(\mu)=\left\{
\begin{array}{ll}
P(T,f)-\mu(f)-h_{\mu}(T) & \textnormal{if}\ \mu\in\mathcal{M}(J,T)
\\ \\
+\infty & \textnormal{if}\
\mu\in\mathcal{M}(J)\verb'\'\mathcal{M}(J,T),
\end{array}
\right.
\end{equation}
  where  $f$ is some parameter belonging to the set $C(J)$  of real-valued
continuous functions on $J$, and  $P(T,\cdot)$, $h_{\cdot}(T)$,
$\mathcal{M}(J,T)$ denote respectively  the  pressure map, the
entropy map, and the set of invariant elements of $\mathcal{M}(J)$.
For
each net $(\nu_\alpha)$ of Borel
probability measures on $\mathcal{M}(J)$ satisfying a large
deviation principle with powers $(t_\alpha)$ and rate function $I^f$, the Varadhan's
theorem yields the following equality,
\begin{equation}\label{equality-of-functional-intro}
 \forall g\in C(J),\ \ \ \ \
\mathbb{L}(\widehat{g}):=\lim t_\alpha\log\nu_\alpha(e^{\widehat{g}/t_\alpha})=
P(T,f+g)-P(T,f),\end{equation} where $\widehat{g}(\mu)=\mu(g)$ for all $\mu\in\mathcal{M}(J)$.
 Our main result
establishes the converse, namely,
(\ref{equality-of-functional-intro}) implies the large deviation
principle with rate function $I^f$ (Theorem \ref{LDP}). From   a well-known result in large deviation theory, the
upper-bounds with $I^f$ follow from  the mere inequality $``\le"$ in (\ref{equality-of-functional-intro}) with moreover
only a upper limit in the L.H.S.; this simple fact allows us to recover the upper bounds proved in \cite{Pollicott-Sridharan-07-JDSGT}, \cite{pol_CMP_96}  (\cf Remark \ref{remark-upper-bounds}).
  The difficult part consists in
proving the lower-bounds. The proof rests on the combining of two
ingredients. The first is a certain entropy-approximation property for
hyperbolic rational maps, namely, any invariant measure can be
approximated weakly$^*$ and in entropy by a sequence of measures,
each one being the unique equilibrium state for some H\"{o}lder
continuous function; the second one is an application of a general
large deviation  result in topological vector spaces, the Baldi's
theorem. More precisely, the net $(\nu_\alpha)$ is considered as acting on the vector space $\widetilde{\mathcal{M}}(J)$ of
signed Borel measures on $J$.
 The sufficient condition for the  large deviation principle in  Baldi's
 theorem involves exposed points and
 exposing hyperplanes  of the  functional $\overline{\mathbb{L}}$, defined with a upper limit  in (\ref{equality-of-functional-intro}) (\cf \S \ref{prelim}). Part of this condition is ensured by
(\ref{equality-of-functional-intro}), and the  remaining is given by
the approximation property; indeed, the statement  $``\mu$ is an
exposed point of the Legendre-Fenchel transform $\overline{\mathbb{L}}^*$ with exposing hyperplane $\widehat{g}"$  is just an
abstract  formulation of  $``\mu$ is the unique equilibrium state
for $f+g"$.

 The main advantage  in having the large deviation principle
 follows from the fact that  the rate function $I^f$ is affine, which  implies that
convex open sets containing some invariant measures are $I^f$-continuity sets (\cf \S \ref{prelim}), and so  we get
  limits on these sets. Furthermore, in view of the form of the rate function (\ref{rate-function-intro}), the above entropy-approximation property is nothing but a continuity property of $I^f$; this  gives a strengthened form of the  large deviation principle  in the sense that the limit on  convex open sets is given up to $\varepsilon$ by $I^f(\mu_{f+g_\varepsilon})$, where $\mu_{f+g_\varepsilon}$ is  the unique equilibrium state for some H\"{o}lder continuous function $f+g_{\varepsilon}$.
  As a consequence, to each of  such large deviation principles and for
  each invariant measure $\mu$, corresponds
    an  expression of the entropy $h_\mu(T)$ in terms of these limits.

  Practically, to each way of obtaining the pressure in the sense of
    (\ref{equality-of-functional-intro}) corresponds a large
    deviation principle.
    This is the case  for every $f\in C(J)$ with the sequence
    \[\nu_{n,f}:=\sum_{y\in
J_n}
p_{n,f}(y)\delta_{\frac{1}{n}(\delta_y+...+\delta_{T^{n-1}y})},\]
where  $J_n$ is a maximal $(\varepsilon,n)$-separated set with
$\varepsilon$ small enough, and
\[p_{n,f}(y)=\frac{e^{f(y)+...+f(T^{n-1}y)}}{\sum_{z\in
J_n}e^{f(z)+...+f(T^{n-1}z)}}.\]
This case is new
and constitutes our first example (\S \ref{subsection-separated-set}).
 It is generic in the sense that it
suffices to replace  $J_n$ either by the set of $n$-preimages or
$n$-periodic points  (with moreover $f$  H\"{o}lder
continuous) in order to get the same results for the
corresponding measures  (\S \ref{pre-per}); this
   improves \cite{Pollicott-Sridharan-07-JDSGT} and \cite{pol_CMP_96} by giving  the lower-bounds.
   The large deviation principle   concerning
    the distribution of   Birkhoff averages with
respect to the measure of maximal entropy (proved in \cite{lop}) is a direct consequence of our  main result (Remark \ref{Birk-general}).
The
expressions of  the entropy corresponding to the above examples are
given by  (\ref{ex-pre-per-eq8.1.0}), (\ref{ex-pre-per-eq8.1.1}),
 (\ref{level-2-birk-general-eq7}).

    Also, by the so-called
    contraction principle  with any $k\in C(J)$ (\cite{dem}, \cite{com-TAMS-03}),
    we get a level-1 large deviation principle  for
    the
   sequence $(\widehat{k}[\nu_n])$ of image measures (Corollary
   \ref{LDP-level-1}). The case  $f=-t\log|T'|$ and $k=\log|T'|$ is studied in detail
   (Theorem \ref{level-1-Lya-theo}); in particular, this yields
    various  formulas  for $h_{\mu_s}(T)$, where $\mu_s$ is  the
 equilibrium state for $-s\log|T'|$
 (\cf (\ref{level-1-pre-per-Lya-eq9}), (\ref{level-1-pre-per-Lya-eq11}),
(\ref{level-1-Birk-general-eq1})).
Some of these results are
analogues of those of \cite{Kel-Now_CMP_92} concerning unimodal
interval maps; it is the case for the existence of a strictly
negative limit in (\ref{level-1-pre-per-Lya-eq10}) with $t=0$ and
(\ref{level-1-Birk-general-eq2}), but we have  more here since
the  bounds are explicit.

We draw attention to the fact that  (except those obtained by contraction as above)   we are  concerned here  only with  level-2 large deviation principles  (\ie in the space  $\cM(J)$),
 and we do not take up  level-1 large deviations (\ie in the real line) with some fixed potential and reference measure (see \cite{Rey-Bellet_Young-08-ETDS28}, \cite{Melbourne_Nicol-08-TAMS360} and references therein for recent developments on this topic).

\subsection{Generalization and the example of the multidimensional full shift}

In the last part of the paper (\S \ref{section-extension}), we observe that the proof   of the main theorem (Theorem \ref{LDP})  works as well  for any dynamical system $(\Omega,\tau)$ (where $\Omega$ is the phase space and $\tau$ the semi-group or group action)  satisfying a similar  entropy-approximation property, and for any  net of Borel probability measures on $\mathcal{M}(\Omega)$ satisfying (\ref{equality-of-functional-intro}) (after obvious changes of notations).
 Indeed,  roughly speaking, $I^f$ is nothing but the Legendre-Fenchel transform of the pressure (Lemma \ref{relation-l*-Hmu})  seen as the  large deviation functional associated to the underlying net of measures,
 so that the above  conditions ensure that the hypotheses of Baldi's theorem are satisfied and  the large deviations follow.
The corresponding result is stated in Theorem \ref{LDP-general}; as for rational maps, it constitutes a strong form of a large deviation principle with rate function $I^f$, since for each invariant measure $\mu$,  $I^f(\mu)$ can be approximated by $I^f(\mu_i)$ where $\mu_i$ is the unique equilibrium state for some potential (Remark \ref{remark-strengthened}). Again here, only  the inequality  $"\le"$ in (\ref{equality-of-functional-intro}) with  a  upper limit in the L.H.S. is required to get   the upper bounds  with  $I^f$.

It turns out that recently  Gurevich and Tempelman (\cite{Gurevich-Tempelman-05-PTRF}, Theorem 1) proved  the entropy-approximation property for the multidimensional full shift. As a consequence, we obtain large deviation principles for nets constructed with  maximal separated sets (resp. periodic configurations), exactly as for rational maps; both results are new (Theorem \ref{LDP-fullshift-atomic}).

\subsection{Important  remark}\label{important-remark}

  Although  some of the large deviation principles  concerning hyperbolic rational maps proved here  can be proved as well using  Kifer's techniques (\cite{Kifer-94-CMP}, \cite{Kifer-90-TAMS}), this is  generally not the case.
Indeed,
for any dynamical system   $(\Omega,\tau)$ (as in \S \ref{section-extension}) and any $f\in C(\Omega)$,
these techniques combined with the  general  version of Lemma \ref{relation-l*-Hmu}  yield the following  result:  If there exists a countably generated   dense  vector  subspace $V\subset C(\Omega)$ such that $f+g$ has a  unique equilibrium state for all $g\in V$, then any net of Borel probability measures on $\mathcal{M}(\Omega)$ such that
\begin{equation}\label{important-remark-eq1}
\forall g\in V,\ \ \ \ \ \mathbb{L}(\widehat{g})=
P^{\tau}(f+g)-P^{\tau}(f),
\end{equation}
satisfies a large deviation principle with rate function $I^f$.
 In the case of hyperbolic rational maps and  under the assumption that $f$ is H\"{o}lder   continuous, it suffices to take $V$ the  space of such functions and  to apply the above result. However, whatever $V$ satisfying the  above  conditions, taking $g=0$ implies  that $f$ has a unique equilibrium state, and so  this method  does not work for general  $f$,  as in  \S \ref{subsection-separated-set}.
   This observation
  is
 valid a fortiori for the multidimensional time setting, as in \S   \ref{multidimensional-fullshift}
 where Theorem \ref{LDP-fullshift-atomic}  works  for every $f\in C(\Omega)$.
 So, the entropy-approximation property,  a purely dynamical one (\ie  independent of the  measures) gives  a  direct way to get large deviations for any net of measures satisfying (\ref{important-remark-eq1}), which furthermore can work when   usual techniques fail.
  It occurs in the two systems of  distinct nature  considered here, and  a  natural question arises about the  generality of such a property; also, it would be interesting to study the relation with the above mentioned Kifer's condition.

 \subsection{Organization }In \S \ref{prelim} we recall basic facts  of  large deviation theory, and in particular Baldi's theorem.  The next two sections  deal with rational maps; \S \ref{general-results} presents  the general results, and \S \ref{section-examples} the  examples. The last section treats the generalization and the example of multidimensional full shift.

\section{Preliminaries}\label{prelim}

 Let $X$ be
a Hausdorff regular  topological space, let $(\nu_\alpha)$ be a net  of
Borel probability measures on $X$, and let $(t_\alpha)$ be a net in $]0,+\infty[$ converging to $0$. We say that $(\nu_\alpha)$ satisfies a \textit{large
deviation principle with powers $(t_\alpha)$} if there exists a
$[0,+\infty]$-valued lower semi-continuous function $I$  on $X$ such
that
\[\limsup t_\alpha\log\nu_{\alpha}(F)\le-\inf_{x\in F}I(x)\le-\inf_{x\in G}I(x)\le
 \liminf t_\alpha\log\nu_\alpha(G)\]
for all closed sets $F\subset X$ and all open sets $G\subset X$ with $F\subset
G$; such a  function $I$ is then unique (called the \textit{rate
function}) and  given  for each $x\in X$ by
\begin{equation}\label{preliminaries-eq2}
-I(x)=\inf_{G\in\mathcal{G}_x}\liminf t_\alpha\log\nu_\alpha(G)=
\inf_{G\in\mathcal{G}_x}\limsup t_\alpha\log\nu_\alpha(G),
\end{equation} where
$\mathcal{G}_x$ is any local basis at $x$.
When the above large deviation principle holds, we have
\begin{equation}\label{preliminaries-eq4}
\lim t_\alpha\log\nu_\alpha(Y)=- \inf_{x\in Y} I(x)
\end{equation}
 for all  Borel sets $Y\subset X$ satisfying $\inf_{x\in\textnormal{Int\ }{Y}} I(x)=\inf_{x\in\overline{Y}}  I(x)$ (where $\textnormal{Int\ }{Y}$ denotes the interior of $Y$), and  we can replace $Y$ by $\textnormal{Int\ }{Y}$ (resp. $\overline{Y}$) in (\ref{preliminaries-eq4}) (such  sets $Y$ are  called \textit{$I$-continuity} sets).

The \textit{large
deviation functional $\overline{\mathbb{L}}$} associated to
$(\nu_\alpha)$ and  $(t_\alpha)$ is defined on the set of $[-\infty,+\infty[$-valued Borel
functions $h$ on $X$ by
\begin{equation}\label{preliminaries-eq5}
\overline{\mathbb{L}}(h)=\limsup t_\alpha\log\nu_\alpha(e^{h/t_\alpha}).
\end{equation}
Note that
$\overline{\mathbb{L}}$ is continuous with respect to  the uniform metric. We
 write $\mathbb{L}(h)$ when the limit exists.  Assume
furthermore that $X$ is a real topological vector space with
topological dual $X^*$, and let $\overline{\mathbb{L}}^*$ denotes
the Legendre-Fenchel transform of $\overline{\mathbb{L}}_{\mid X^*}$ (this restriction is the so-called "generalized log-moment generating function").
 An element
 $x\in X$ is an \textit{exposed point} of  $\overline{\mathbb{L}}^*$ if there exists
  $\lambda\in X^*$ (called
 \textit{exposing hyperplane})
  such that
 \[
 \forall y\neq x,\ \ \ \ \ \lambda(x)-\overline{\mathbb{L}}^*(x)>
 \lambda(y)-\overline{\mathbb{L}}^*(y).
 \]

 The main tool
to derive the large deviation lower bounds is the following
classical result of Baldi (\cite{bal}, \cite{dem}; see also
\cite{com3} for a strengthened version). We shall apply it with
$X=\widetilde{\mathcal{M}}(\Omega)$ (where $\Omega$ is the phase space of the system) by showing   that the approximation property mentioned in \S \ref{intro}  implies the condition on exposed points.
 We recall that  $(\nu_\alpha)$ is said to be
\textit{exponentially tight with respect to $(t_\alpha)$}
 if for each real $M$ there exists a compact
 $K_M\subset X$  such that
$\limsup t_\alpha\log\nu_\alpha(X\verb'\'K_M)<M$; our nets of measures are trivially exponentially tight since
they are constituted by measures supported by
 the compact set $\mathcal{M}(\Omega)$ (in particular, the boundness  condition  on $\overline{\mathbb{L}}_{\mid X^*}$ is always satisfied).

\begin{theorem}\label{Baldi-theorem}\textnormal{(\textbf{Baldi})}
\label{Baldi} Let $X$ be  a  real  Hausdorff topological vector
space and assume that $(\nu_\alpha)$ is exponentially tight with respect to $(t_\alpha)$. Let
$\mathcal{E}$ be the set of exposed points $x$ of \
$\overline{\mathbb{L}}^*$ for which there is an exposing hyperplane
$\lambda_x$ such that $\mathbb{L}(\lambda_x)$ exists and
$\overline{\mathbb{L}}(c\lambda_x)<+\infty$ for some $c>1$. If
$\inf_{G}\overline{\mathbb{L}}^*=\inf_{G\cap\mathcal{E}}\overline{\mathbb{L}}^*$ for
all open sets $G\subset X$, then $(\nu_\alpha)$ satisfies a large deviation
principle with powers $(t_\alpha)$ and  rate function $\overline{L}^*$.
\end{theorem}

\section{Hyperbolic rational maps and Baldi's  theorem}\label{general-results}

Throughout this section and the next one, $T$ is a hyperbolic  rational map of degree
$d\ge 2$ (\ie  expanding on its Julia set $J$), and $C(J)$ the set
of real-valued continuous functions on $J$ provided with the uniform topology.
 We denote by $\widetilde{\mathcal{M}}(J)$ the vector space of  signed Borel measures on
$J$ provided with the weak$^*$-topology, and by $\mathcal{M}(J)$
(resp. $\mathcal{M}(J,T)$)  the set of Borel probability measures on
$J$ (resp.
  $T$-invariant elements of $\mathcal{M}(J)$) endowed with the induced topology.
  For any $g\in C(J)$,   $P(T,g)$,
  $\mathcal{M}_g(J,T)$ and $h_\mu(T)$ stand for the topological pressure
of $g$, the set of equilibrium states for $g$, and   the
measure-theoretic entropy of $T$ with respect to $\mu$,
respectively; we recall that  $\mathcal{M}_g(J,T)$ has a unique
element when $g$ is H\"{o}lder continuous. We denote by $\widehat{g}$
the map defined on $\widetilde{\mathcal{M}}(J)$ by
$\widehat{g}(\mu)=\mu(g)$.
For each $f\in C(J)$ we define a map  $Q_f$   on $C(J)$ by
 \[Q_f(g)=P(T,f+g)-P(T,f),\]
and note that  $Q_f$ is real-valued, convex,  and continuous with respect to the uniform metric. The
  Legendre-Fenchel transform ${Q_f}^*$ of $Q_f$ will
 appear in the sequel as the rate function of our level-2 large
 deviation principles; note that  ${Q_f}^*$ vanishes exactly
 on $\mathcal{M}_f(J,T)$, as shows the following lemma.

\begin{lemma}\label{relation-l*-Hmu}
We have
\begin{displaymath}
{Q_f}^*(\mu)=\left\{
\begin{array}{ll}
P(T,f)-\mu(f)-h_{\mu}(T) &\ \  if\ \mu\in\mathcal{M}(J,T)
\\ \\
+\infty & \ \ if\
\mu\in\widetilde{\mathcal{M}}(J)\verb'\'\mathcal{M}(J,T).
\end{array}
\right.
\end{displaymath}
For each pair of functions  $f,g$ in $C(J)$ and each $\mu\in\mathcal{M}(J)$, $\mu$ is an equilibrium state for $f+g$ if and only if $Q_f(g)=\mu(g)-{Q_f}^*(\mu)$.
\end{lemma}

\proof
  Put
\begin{displaymath}
U(\mu)=\left\{
\begin{array}{ll}
-\mu(f)-h_{\mu}(T) & \textnormal{if}\ \mu\in\mathcal{M}(J,T)
\\ \\
+\infty & \textnormal{if}\
\mu\in\widetilde{\mathcal{M}}(J)\verb'\'\mathcal{M}(J,T).
\end{array}
\right.
\end{displaymath}
 We have
\[P(T,f+g)=\sup_{\mu\in\mathcal{M}(J,T)}\{\mu(f+g)+h_\mu(T)\}=
\sup_{\mu\in\widetilde{\mathcal{M}}(J)}\{\mu(g)-U(\mu)\},\] and
since the entropy map is affine upper semi-continuous, $U$ is convex
lower semi-continuous and $]-\infty,+\infty]$-valued.
 By the duality theorem, $U$ is the Legendre-Fenchel
transform of the map $g\rightarrow P(T,f+g)$, that is for each
$\mu\in\widetilde{\mathcal{M}}(J)$,
\[
\begin{split}
U(\mu) & =  \sup_{g\in C(J)}\{\mu(g)-P(T,f+g)\}=\sup_{g\in C(J)}\{\mu(g)-P(T,f)-Q_f(g)\}
\\ & =-P(T,f)+\sup_{g\in
C(J)}\{\mu(g)-Q_f(g)\}=  -P(T,f)+{Q_f}^*(\mu),
\end{split}
\]
 which proves the first assertion. The last assertion   follows from the equalities
\[
\begin{split}
 P(T,f+g)-h_\mu(T)-\mu(f+g) &
 =
Q_f(g) + P(T,f) -h_\mu(T) -\mu(f+g)
\\ &
 =
Q_f(g) + {Q_f}^* (\mu)-\mu(g).
\end{split}
\]
\endproof

 The following approximation property  is proved in \cite[Theorem 8]{lop}.

 \begin{theorem}\label{Lopes-theorem}\textnormal{(\textbf{Lopes})}
 For each $\mu\in\mathcal{M}(J,T)$ there exists a sequence
$(k_i)$ of H\"{o}lder continuous functions  on $J$ such that the
sequence $(\mu_i)$ of their  respective equilibrium states satisfies
 $\lim\mu_i=\mu$ and
 $\lim h_{\mu_i}(T)=h_\mu(T)$.
 \end{theorem}

The main result of this section is the following theorem, where  the  large deviation principle is obtained by a direct application of Baldi's theorem, once observed that the condition  on exposed points follows  from  Theorem \ref{Lopes-theorem}.

\begin{theorem}\label{LDP}
Let $f\in C(J)$,
let $(\nu_\alpha)$ be a net of Borel probability measures on
$\mathcal{M}(J)$, let $(t_\alpha)$ be a net in $]0,+\infty[$ converging to $0$,  and assume that
\[\lim t_\alpha\log\nu_\alpha(e^{\widehat{g}/t_\alpha})=Q_f(g)
\]for all $g$ in a dense subset of $C(J)$.
\begin{nitemize}
\item[\textnormal{a)}] The net $(\nu_\alpha)$ satisfies a large deviation
principle with power $(t_\alpha)$ and  rate function
\begin{displaymath}
I^f(\mu)=\left\{
\begin{array}{ll}
P(T,f)-\mu(f)-h_{\mu}(T) & \ \ if\ \mu\in\mathcal{M}(J,T)
\\ \\
+\infty & \ \ if\
\mu\in{\mathcal{M}}(J)\verb'\'\mathcal{M}(J,T).
\end{array}
\right.
\end{displaymath}
 Moreover,
  for each convex
 open  set $G\subset\mathcal{M}(J)$ containing some invariant measure  we have
\begin{equation}\label{LDP-eq-2}
\lim t_\alpha\log\nu_\alpha(G) =\lim t_\alpha\log\nu_\alpha(\overline{G})=
-\inf_{\mu\in
\overline{G}}I^f(\mu)= -\inf_{\mu\in G\cap\cE'}I^f(\mu),
\end{equation}
where $\cE'$ is the set of  equilibrium states of all H\"{o}lder continuous functions on $J$. In particular,
  for each $\mu\in\mathcal{M}(J,T)$ and each  convex
local basis  $\mathcal{G}_\mu$   at $\mu$ we obtain
\[h_{\mu}(T)=P(T,f)-\mu(f)+\inf\{\lim t_\alpha\log\nu_\alpha(G):
G\in\mathcal{G}_\mu\}.\]
\item[\textnormal{b)}]  Each  limit point of $(\nu_\alpha)$ has its support included in $\mathcal{M}_{f}(J,T)$; in particular,
$\lim\nu_\alpha=\delta_{\mu_f}$ when $f$ has a unique equilibrium sate $\mu_f$.
\end{nitemize}
\end{theorem}

\proof
a)
We consider  $(\nu_\alpha)$ as a net of Borel
probability measures on $\widetilde{\mathcal{M}}(J)$, and note that
the corresponding large deviation functional (with slight abuse of
notation we denoted also by $\overline{\mathbb{L}}$) satisfies
$\overline{\mathbb{L}}(\widehat{\cdot})=\overline{\mathbb{L}}(\widehat{\cdot}_{\mid \mathcal{M}(J)})$.
Since $\sup_{J}|g|=\sup_{\cM(J)}|\widehat{g}|$,  the maps  $g\mapsto\overline{\mathbb{L}}(\widehat{g})-Q_f(g)$ and $g\mapsto\underline{\mathbb{L}}(\widehat{g})-Q_f(g)$ are  continuous (where $\underline{\mathbb{L}}$ is defined replacing the upper limit by a lower limit in (\ref{preliminaries-eq5})),
and so
the general  hypothesis implies the existence of ${\mathbb{L}}(\widehat{g})$ with the equality
  ${\mathbb{L}}(\widehat{g})=Q_f(g)$ for all $g\in C(J)$.
  Consequently, we have
  \begin{equation}\label{LDP-eq-6}
  {{\mathbb{L}}}^*={Q_f}^*,
  \end{equation} and  the large deviation
upper-bounds in  $\widetilde{\mathcal{M}}(J)$ with the function
${{\mathbb{L}}}^*$  follow from a well-known result in topological
vector spaces, namely, Theorem 4.5.3 of \cite{dem} (\cf Remark \ref{remark-upper-bounds}). Let $\mathcal{E}$ denote the set of exposed points of ${\mathbb{L}}^*$, and note that $\mathcal{E}$  coincides with the set denoted by the same symbol in Theorem \ref{Baldi} since $\mathbb{L}_{\mid \widetilde{\mathcal{M}}(J)^*}$ is here real-valued.
For every  $\mu\in\mathcal{M}(J)$,  (\ref{LDP-eq-6}) and the last assertion of Lemma \ref{relation-l*-Hmu} show  that  $\mu$  is the unique equilibrium state
for $f+g$ if and only if  $\mu\in\mathcal{E}$ with exposing hyperplane $\widehat{g}$.
Let $\cE'$ denote the subset of $\cE$ constituted by the equilibrium states of all H\"{o}lder continuous functions.
Putting
$k_i=f+g_i$ in  Theorem \ref{Lopes-theorem}, it follows from the expression  of ${Q_f}^*$ given by
 Lemma \ref{relation-l*-Hmu}  that for each
$\mu\in\mathcal{M}(J,T)$, there exists a sequence $(f+g_i)$ in $C(J)$ and a sequence   $(\mu_i)$ in $\mathcal{M}(J,T)$ such that
\begin{itemize}
\item[(i)] $\mu_i\in\mathcal{E}'$ (with exposing  hyperplane $\widehat{g_i}$, or equivalently $\mu_i$ is the
unique equilibrium state for $f+g_i$);
\item[(ii)] $\lim\mu_i=\mu$;
\item[(iii)] $\lim {{\mathbb{L}}}^*(\mu_i)={{\mathbb{L}}}^*(\mu)$.
\end{itemize}
Since $\cE'\subset\mathcal{E}\subset\mathcal{M}(J,T)$ and ${\mathbb{L}}^*$ is infinite-valued   outside  $\mathcal{M}(J,T)$ we have for each open set $G\subset\mathcal{M}(J)$,
\begin{equation}\label{LDP-eq-8}
\inf_{G}{\mathbb{L}}^*=\inf_{G\cap\mathcal{M}(J,T)}{\mathbb{L}}^*
\le\inf_{G\cap\mathcal{E}}{\mathbb{L}}^*
\le\inf_{G\cap\mathcal{E}'}{\mathbb{L}}^*,
\end{equation}
and the  properties (i)-(iii) show that both above inequalities are equalities.
 Consequently, all the hypotheses of Theorem \ref{Baldi} are fulfilled, and so
  $(\nu_\alpha)$ satisfies a large deviation principle in
$\widetilde{\mathcal{M}}(J)$ with powers $(t_\alpha)$ and rate  function ${{\mathbb{L}}}^*$.
Since $\mathcal{M}(J)$  is closed in $\widetilde{\mathcal{M}}(J)$,
the large deviation principle holds  in
$\mathcal{M}(J)$ with rate function ${{{\mathbb{L}}}^*}$$_{\mid
\mathcal{M}(J)}$ (\cite{dem}, Lemma 4.1.5); this proves the first
assertion since $I^f={{{\mathbb{L}}}^*}$$_{\mid
\mathcal{M}(J)}$ by (\ref{LDP-eq-6}) and Lemma \ref{relation-l*-Hmu}.
 Let
$G\subset\mathcal{M}(J)$ be  a convex open set containing some
invariant measure $\mu'$.  Let $\mu\in \overline{G}$ satisfying
$I^f(\mu)=\inf_{\overline{G}}I^f$, and suppose that
$I^f(\mu)<\inf_{G}I^f$; in particular
$\mu\in\mathcal{M}(J,T)\cap\overline{G}\verb'\'G$. Let $\lambda_n$
be a sequence in $]0,1[$ converging to $0$, put
$\mu_n=\lambda_n\mu'+(1-\lambda_n)\mu$, and note that $\mu_n\in G$
(\cite{sch}, pp. 38). We have $\lim\mu_n=\mu$ and $\lim
I^f(\mu_n)=I^f(\mu)$ since $I^f$ is affine and real-valued on
$\mathcal{M}(J,T)$, which gives the contradiction; therefore $\inf_G
I^f=\inf_{\overline{G}}I^f$, and  since (\ref{LDP-eq-8}) holds with equalities
  we get
\[\inf_{\overline{G}}I^f=\inf_{G\cap\cE'}I^f.\] The two last
assertions  are direct consequences of the above  equality and  the
large deviation principle (\cf (\ref{preliminaries-eq4}),  (\ref{preliminaries-eq2})).

b)
Let $(\nu_\beta)$ be a subnet of $(\nu_\alpha)$ converging to some $\nu$, let $\mu\in\mathcal{M}(J,T)\verb'\'\mathcal{M}_f(J,T)$, and let
 $G$ be an open set satisfying
 \[\mu\in G\subset\overline{G}\subset\mathcal{M}(J,T)\verb'\'\mathcal{M}_f(J,T).\]
 The large deviation upper bounds yields
 \[\limsup t_\alpha\log\nu_\alpha(G)\le\limsup t_\alpha\log\nu_\alpha(\overline{G})<0,\]hence \[\lim\nu_\beta(G)=\nu(G)=0,\] which shows that $\textnormal{supp}(\nu) \subset\mathcal{M}_f(J,T)$.
\endproof

\begin{corollary}\label{LDP-level-1}
Under the hypotheses of Theorem \ref{LDP}, for each $k\in C(J)$ the
 net of image measures $(\widehat{k}[\nu_\alpha])$ satisfies a
large deviation principle in $\mathbb{R}$ with powers $(t_\alpha)$ and convex rate function
\[\forall x\in\mathbb{R},\ \ \ \ \ \ I^f_k(x)=\inf_{\{\mu\in\mathcal{M}(J):\mu(k)=x\}}
 I^f(\mu).\] Explicitly, we have
 \[\limsup t_\alpha\log\nu_\alpha\{\mu\in\mathcal{M}(J):\mu(k)\in C\}\le -\inf_{x\in C}I^f_k(x)
\] for all closed sets  $C\subset\mathbb{R}$,
  and
\[\liminf t_\alpha\log\nu_\alpha\{\mu\in\mathcal{M}(J):\mu(k)\in U\}\ge -\inf_{x\in U}I^f_k(x)\]
 for all open sets $U\subset\mathbb{R}$.
 Moreover, for each $\mu\in\mathcal{M}(J,T)$
and  each $\varepsilon>0$ small enough, we have
\begin{equation}\label{LDP-level-1-eq0}
\lim t_\alpha\log\nu_\alpha(G_{k,\mu,\varepsilon})=\lim t_\alpha\log\nu_\alpha
(\overline{G_{k,\mu,\varepsilon}})=-\inf_{G_{k,\mu,\varepsilon}}
I^f=-\inf_{\overline{G_{k,\mu,\varepsilon}}}I^f,
\end{equation}
where
  $G_{k,\mu,\varepsilon}=\{\mu'\in\mathcal{M}(J):
|\mu'(k)-\mu(k)|>\varepsilon\}$.
\end{corollary}

\proof
 The large deviation principle with rate function $I^f_k$ follows from the contraction principle (\cite{dem},
Theorem 4.2.1) applied  to  $(\nu_\alpha)$ with  $\widehat{k}$.
 For each pair of reals $x_1,x_2$
 and  each $\beta\in]0,1[$ we have
\[I^f_k(\beta x_1+(1-\beta)x_2)=\inf\{I^f(\mu):
\mu\in{\mathcal{M}}(J,T),\mu(k)=\beta
x_1+(1-\beta)x_2\}\]
\[\le\inf\{I^f(\beta\mu_1+(1-\beta)\mu_2):\mu_1\in\mathcal{M}(J,T),
\mu_2\in\mathcal{M}(J,T),\mu_1(k)=x_1,\mu_2(k)=x_2\}\]
\[=\inf\{\beta I^f(\mu_1)+(1-\beta)I^f(\mu_2):\mu_1\in\mathcal{M}(J,T),
\mu_2\in\mathcal{M}(J,T),\mu_1(k)=x_1,\mu_2(k)=x_2\}\]
\[\le\beta I^f_k(x_1)+(1-\beta)I^f_k(x_2)\]
 and hence $I^f_k$ is convex.
For each $\mu\in\mathcal{M}(J,T)$ and each $\delta>0$ we put
\[ G_{1,\delta}
=
\{ \mu' \in \mathcal{M}(J) :  \mu'(k)-\mu(k) > \delta \}
\]
and
\[
G_{2,\delta}
=
\{ \mu' \in \mathcal{M}(J) : \mu'(k)-\mu(k) < -\delta \},
\]
and note that
\begin{equation}\label{LDP-level-1-eq3}
\forall \delta'<\delta,\ \ \ \ {G_{j, \delta'}} \supset \overline{G_{j, \delta}},\ \ \ \ \ \ \ \ (j \in \{1, 2 \}).
\end{equation}
First assume that   $(G_{1,\delta}\cup G_{2,\delta})\cap\mathcal{M}(J,T)\neq\emptyset$  for  some $\delta>0$. If $G_{1,\delta}\cap G_{2,\delta}$ contains some invariant measure then   (\ref{LDP-level-1-eq0}) with $\varepsilon=\delta$ follows  from  (\ref{LDP-eq-2}) applied to $G_{1,\delta}\cup G_{2,\delta}$. Clearly the same  holds for all $\varepsilon \in ]0, \delta]$ by (\ref{LDP-level-1-eq3}).
If $G_{j,\delta}\cap\mathcal{M}(J,T)=\emptyset$  for some $j \in \{1, 2 \}$ (say $j=1$), then  either  $G_{1,\delta'}\cap\mathcal{M}(J,T)\neq\emptyset$
 for all  $\delta'<\delta$,  and we fall in the  preceding  case with $\delta'$ in place of $\delta$;
 either  $\overline{G_{1,\delta}}\cap\mathcal{M}(J,T)=\emptyset$  and  the  large deviation upper bounds  yields  $\limsup t_\alpha\log\nu_\alpha(\overline{G_{1,\delta}})=-\infty$, and  (\ref{LDP-level-1-eq0}) follows from (\ref{LDP-eq-2}) applied to $G_{2,\delta}$.
When
$(G_{1,\delta}\cup G_{2,\delta})\cap\mathcal{M}(J,T)=\emptyset$  for all $\delta > 0$, from (\ref{LDP-level-1-eq3})
 we obtain
\[\forall\delta>0,\ \ \ \ \overline{G_{1,\delta}\cup G_{2,\delta}}\cap\mathcal{M}(J,T)=\emptyset, \]
 and (\ref{LDP-level-1-eq0}) follows from  the   upper bounds applied to $\overline{G_{1,\delta}\cup G_{2,\delta}}$ for all $\delta>0$ (note that in this case, (\ref{LDP-level-1-eq0}) takes an  infinite value, and  $\widehat{k}_{\mid \mathcal{M}(J,T)}$ is the constant function equals to $\mu(k)$).
\endproof

For each real $t$, we put  $k_t=-t\log|T'|$  (\ie
 $\widehat{{k_{-1}}}_{\mid\mathcal{M}(J,T)}=\chi$ is the Lyapunov
 map), and we write $P(t)$ for
$P(T,k_t)$. The map $t\mapsto P(t)$ is real analytic, strictly decreasing, and strictly
convex  when $T$ is not conjugate to
$z\mapsto z^{\pm d}$ (in this last case $P(t)=\log d(1-t)$). The map
$k_t$  has a unique equilibrium state, which we
  denote in what follows simply by $\mu_t$, in place of $\mu_{k_t}$ (\cite{Rue-82-ETDS2}, \cite{Keller98}).
   We put $\chi_{\inf}=\inf\{\chi(\mu):\mu\in\mathcal{M}(J,T)\}$,
   $\chi_{\sup}=\sup\{\chi(\mu):\mu\in\mathcal{M}(J,T)\}$,
   and recall that $\chi_{\inf}=\inf\{-P'(t):t\in\mathbb{R}\}>0$
   and $\chi_{\sup}=\sup\{-P'(t):t\in\mathbb{R}\}$.
In the following,  we specify  Corollary \ref{LDP-level-1} by taking
$f=k_t$ and $k=k_{-1}$.

\begin{theorem}\label{level-1-Lya-theo}
Assume that  the hypotheses of Theorem \ref{LDP} hold with $f=k_t$
for some real $t$. Then the net $(\widehat{k_{-1}}[\nu_\alpha])$
satisfies a large deviation principle with  rate function
\begin{equation}\label{level-1-Lya-theo-eq3}
I^{k_t}_{k_{-1}}(x)=\left\{
\begin{array}{ll}
0 & \ \ if\ x=\log d
\\ \\
+\infty & \ \ if\ x\neq \log d
\end{array}
\right.
\end{equation}
when  $T$ is  conjugate to $z\mapsto z^{\pm d}$, and
\begin{equation}\label{level-1-Lya-theo-eq4}
I^{k_t}_{k_{-1}}(x)=\left\{
\begin{array}{ll}
P(t)+tx-h_{\mu_{s_x}}(T) & if\ \chi_{\inf}< x <\chi_{\sup}
\\ \\
+\infty & otherwise
\end{array}
\right.
\end{equation}in all others cases,
where $s_x$ is the unique   real such that $\chi(\mu_{s_x})=x$;
${I^{k_t}_{k_{-1}}}$ is then strictly convex  on
$]\chi_{\inf},\chi_{\sup}[$, and   essentially smooth.
  Moreover, we have
\[
\lim\nu_\alpha(\widehat{k_{-1}})=\chi(\mu_t).
\]
\end{theorem}

\proof
For each real $s$   let  $u_s$ be the map defined on $\mathbb{R}$ by
$u_s(x)=sx$.
 We have
\[
L(s):=\lim t_\alpha\log\widehat{k_{-1}}[\nu_\alpha](e^{
u_s/t_\alpha})=\lim t_\alpha\log\nu_\alpha(e^{u_s\circ\widehat{k_{-1}}/t_\alpha})=
\mathbb{L}(u_s\circ\widehat{k_{-1}})
\]
\[=\sup_{\mu\in\mathcal{M}(J)}
\{u_s\circ\widehat{k_{-1}}(\mu)-I^{f}(\mu)\}=\sup_{\mu\in\mathcal{M}(J,T)}
\{\mu(sk_{-1})-P(T,k_t)+\mu(k_t)+h_\mu(T)\}\]
\[=P(t-s)-P(t),\] where the existence of the limits and  the
fourth equality follow from  the large deviation principle for
$(\nu_\alpha)$ and the Varadhan's theorem applied to the bounded
continuous function $u_s\circ\widehat{k_{-1}}_{\mid \mathcal{M}(J)}$ (\cf \cite{dem}, \cite{com-TAMS-03}).
The map $L$ is then differentiable on $\mathbb{R}$, and  consequently
  $I^{k_t}_{k_{-1}}=L^*$ by G\"{a}rtner-Ellis theorem (\cite{dem}). This proves
(\ref{level-1-Lya-theo-eq3}) since $P(t)=\log d(1-t)$ when $T$ is
conjugate to $z\mapsto z^{\pm d}$. Assume that $T$ is not conjugate
to $z\mapsto z^{\pm d}$. We have
\begin{equation}\label{level-1-Lya-theo-eq6}
\forall\mu\in\mathcal{M}(J,T),\ \ \ \ \ I^{k_{t}}(\mu)=P(t)+t\chi(\mu)-h_{\mu}(T)
\end{equation}
 and
\[
P(s_{\chi(\mu)})=  -s_{\chi(\mu)}\chi(\mu_{s_{\chi(\mu)}})+h_{\mu_{s_{\chi(\mu)}}}(T)=
-s_{\chi(\mu)}\chi(\mu)+h_{\mu_{s_{\chi(\mu)}}}(T)
\]
\[ \ge
-s_{\chi(\mu)}\chi(\mu)+h_\mu(T),
\]hence
\begin{equation}\label{level-1-Lya-theo-eq6.1}
\forall\mu\in\mathcal{M}(J,T),\ \ \ \ \ h_{\mu}(T)\le
h_{\mu_{s_{\chi(\mu)}}}(T).
\end{equation}
Combining (\ref{level-1-Lya-theo-eq6}) and (\ref{level-1-Lya-theo-eq6.1}) we get
\[
\forall x\in\mathbb{R},\ \ \
 \ \ I^{k_{t}}_{k_{-1}}(x)=
 \inf_{\mu\in\mathcal{M}(J,T),\chi(\mu)=x}I^{k_{t}}(\mu)=
\inf_{s\in\mathbb{R},\chi(\mu_s)=x}I^{k_{t}}(\mu_s)
\]
\[=P(t)+\inf_{s\in\mathbb{R},\chi(\mu_s)=x} \{t\chi(\mu_s)-h_{\mu_{s_{\chi(\mu_s)}}}(T)\},\]
 which gives  (\ref{level-1-Lya-theo-eq4}).
 If  $L^*$ has a subgradient  at $x$, then it is unique
(namely, the  real $s$ such that $L'(s)=x$), hence $L^*$ is strictly
convex on the interior of its effective domain and essentially
smooth (\cite{roc}, Corollary 26.3.1). In all cases,
 the
 last assertion  follows from Theorem \ref{LDP} b).
\endproof

\begin{remark}\label{remark-upper-bounds}
In the proof of Theorem \ref{LDP}, in order to get the upper bounds with rate function ${Q_f}^*$, in view of Theorem 4.5.3 of \cite{dem} we  only need the  inequality
 $\overline{{\mathbb{L}}}^*\ge{Q_f}^*$  in place of (\ref{LDP-eq-6}), and therefore the inequality
$\overline{\mathbb{L}}(\widehat{g})\le
Q_f(g)$ for all $g$ in a dense subset of $C(J)$ is sufficient; in particular, the existence of $\mathbb{L}(g)$  is not necessary. Also, part b) of Theorem \ref{LDP} follows from the upper bounds.
\end{remark}

\section{Examples}\label{section-examples}

In this section we apply the preceding  results to various sequences
 of measures on $\mathcal{M}(J)$ with powers  $(1/n)$. Our first example is analogue to
those of \cite{pol_CMP_96} and \cite{Pollicott-Sridharan-07-JDSGT}  constructed by
means of preimages and periodic points, respectively. The change
consists in replacing these points by the elements of maximal
separated sets, and allowing $f$ to be  any element of $C(J)$; note that the  techniques based on  \cite{Kifer-90-TAMS} do not work here  when $f$ has more than one equilibrium state (\cf   \S \ref{important-remark}).
All
the results obtained hold verbatim for preimages and periodic
points, when $f$ is H\"{o}lder continuous (\cf \S \ref{pre-per}).

In order  to handle the above atomic examples as well as the case of distribution of
Birkhoff averages with respect to the measure of maximal entropy (Remark \ref{Birk-general}), we shall consider the
following scheme. For each $y\in J$ and each integer
$n\ge 1$ we define the measure
$\mu_{y,n}=\frac{1}{n}(\delta_y+...+\delta_{T^{n-1}y})$, and for
each $k\in C(J)$  we put
$S_n(k)(y)=k(y)+...+k(T^{n-1}y)$. We shall fix some $f\in C(J)$ and
consider a suitable sequence $(J_n)$ of subspaces of $J$, each one
provided with a Borel probability measure $p_{n,f}$, and set
$\nu_{n,f}=W_n[p_{n,f}]$, where $W_n$ is the $\mathcal{M}(J)$-valued random
variable on $J_n$ defined by $W_n(y)=\mu_{y,n}$. The large deviation principles  are obtained from Theorem \ref{LDP},  once checked that $(\nu_{n,f})$ fulfils the general hypothesis.

\subsection{Separated sets, $f\in C(J)$}\label{subsection-separated-set}

Let $f\in C(J)$, let $\varepsilon_0$ be the  expansivity constant
for  $J$, let
 $\varepsilon<\varepsilon_0/2$,
 let $J_n$
 be a maximal
$(\varepsilon,n)$-separated set,  and let $p_{n,f}$ be the probability
measure having mass $p_{n,f}(y)=\frac{e^{S_n(f)(y)}}{\sum_{z\in
J_n}e^{S_n(f)(z)}}$ at each $y\in J_n$, so that $\nu_{n,f}=\sum_{y\in
J_n} p_{n,f}(y)\delta_{\mu_{y,n}}$ (note that
$p_{n,f}=\textnormal{Card}(J_n)^{-1}$ when $f=0$).

 \subsubsection{Level-2 large deviation principles}\label{subsection-separated-set-level-2}
Direct computations yield for each $g\in C(J)$,
\begin{equation}\label{ex-pre-per-eq3}
 \frac{1}{n}\log\nu_{n,f}(e^{n\widehat{g}})=
 \frac{1}{n}\log\sum_{y\in J_n}
e^{S_n(f+g)(y)}-  \frac{1}{n}\log\sum_{y\in J_n}
e^{S_n(f)(y)},\end{equation}  and since by definition of the pressure
\begin{equation}\label{ex-pre-per-eq4}
\lim\frac{1}{n}\log\sum_{y\in J_n} e^{S_n(k)(y)}=P(T,k)
\end{equation}
for all $k\in C(J)$ (\cite{Rue-78}), we get  $\mathbb{L}(\widehat{g})=Q_f(g)$ by taking the limit in
(\ref{ex-pre-per-eq3}), and the hypothesis of Theorem \ref{LDP}
is fulfilled; consequently all the conclusions of this theorem hold.
 Explicitly, the large deviation principle means that
\begin{equation}\label{ex-pre-per-eq6}
\limsup\frac{1}{n}\log\sum_{y\in J_n,\mu_{y,n}\in
F}p_{n,f}(y)\le-\inf_{\mu\in
F\cap\mathcal{M}(J,T)}\{P(T,f)-\mu(f)-h_{\mu}(T)\}
\end{equation}
 for all closed sets
$F\subset\mathcal{M}(J)$, and
\begin{equation}\label{ex-pre-per-eq7}
\liminf\frac{1}{n}\log\sum_{y\in J_n,\mu_{y,n}\in
G}p_{n,f}(y)\ge-\inf_{\mu\in
G\cap\mathcal{M}(J,T)}\{P(T,f)-\mu(f)-h_{\mu}(T)\}
\end{equation} for all open sets
$G\subset\mathcal{M}(J)$.
 When $G\subset\mathcal{M}(J)$ is a convex open set  containing some
invariant measure, the lower limit in (\ref{ex-pre-per-eq7}) is a
limit and  the inequality is an equality; furthermore,  there exists
a sequence of invariant measures $(\mu_m)$, each of which being the unique equilibrium state for some H\"{o}lder continuous function, converging weakly$^*$ and in entropy to some invariant
measure $\mu\in\overline{G}$ realizing the infimum of $I^f$ on
$\overline{G}$ and $G$, \ie,
\begin{equation}\label{ex-pre-per-eq8}
\lim\frac{1}{n}\log\sum_{y\in J_n,\mu_{y,n}\in
G}p_{n,f}(y)=\lim\frac{1}{n}\log\sum_{y\in J_n,\mu_{y,n}\in
\overline{G}}p_{n,f}(y)
\end{equation}
\[=\mu(f)+h_{\mu}(T)-P(T,f)=\lim \mu_m(f)+h_{\mu_m}(T)-P(T,f).\]
The  R.H.S. of
 (\ref{ex-pre-per-eq6}) is
 strictly negative when $F\cap\mathcal{M}_f(J,T)=\emptyset$.
 If furthermore we assume that $F$ is convex with nonempty interior, since $F$ is necessarily regular (\cite{sch}, pp. 38), we can apply
 (\ref{ex-pre-per-eq8}) with $G=\textnormal{Int}(F)$  so  that
 the
upper limit  in (\ref{ex-pre-per-eq6}) is a
limit and  the inequality is an equality.
 The last assertion of Theorem \ref{LDP} a) yields for each invariant measure $\mu$,
\begin{equation}\label{ex-pre-per-eq8.1}
h_{\mu}(T)=P(T,f)-\mu(f)+\lim_{\varepsilon\rightarrow
0}\lim\frac{1}{n}\log\sum_{y\in
J_n,\rho(\mu,\mu_{y,n})<\varepsilon}p_{n,f}(y),
\end{equation}
where $\rho$ is  any distance on $\mathcal{M}(J)$ compatible with the
weak$^*$-topology,  and for which the open balls are
 convex. Since
\[\log\sum_{y\in
J_n,\rho(\mu,\mu_{y,n})<\varepsilon}p_{n,f}(y)=\log\sum_{y\in
J_n,\rho(\mu,\mu_{y,n})<\varepsilon}
{e^{S_n(f)(y)}}-\log{\sum_{z\in
J_n}e^{S_n(f)(z)}},\]
and
\[\lim\frac{1}{n}\log{\sum_{z\in
J_n}e^{S_n(f)(z)}}=P(T,f)\] by
  (\ref{ex-pre-per-eq4}),  from   (\ref{ex-pre-per-eq8.1})   we get  the following expression for the measure-theoretic
entropy of $\mu$,
 \begin{equation}\label{ex-pre-per-eq8.1.0}
h_{\mu}(T)=-\mu(f)+\lim_{\varepsilon\rightarrow
0}\lim\frac{1}{n}\log\sum_{y\in
J_n,\rho(\mu,\mu_{y,n})<\varepsilon}{e^{S_n(f)(y)}},
\end{equation}
and   taking $f=0$,
\begin{equation}\label{ex-pre-per-eq8.1.1}
h_{\mu}(T)=\lim_{\varepsilon\rightarrow
0}\lim\frac{1}{n}\log\textnormal{Card\ } \{y\in
J_n:\rho(\mu,\mu_{y,n})<\varepsilon\}.
\end{equation}
When  $f$ has a unique  equilibrium state $\mu_f$ we obtain from Theorem \ref{LDP} b),
 \begin{equation}\label{ex-pre-per-eq10}
 \lim\sum_{y\in
J_n} p_{n,f}(y)\mu_{y,n}=\mu_f.
\end{equation}

 \subsubsection{Level-1 large deviation principles}\label{level-1-pre-per}
 By applying  Corollary
\ref{LDP-level-1} to the sequence $(\nu_{n,f})$ and  any $k\in C(J)$, we
 obtain the following large deviation results.
\begin{equation}\label{level-1-pre-per-eq1}
\limsup\frac{1}{n}\log\sum_{y\in J_n,\frac{S_n(k)(y)}{n}\in
C}p_{n,f}(y)\le-\inf_{\mu\in\mathcal{M}(J,T),\mu(k)\in
C}\{P(T,f)-\mu(f)-h_{\mu}(T)\}
\end{equation}
 for all closed sets
$C\subset\mathbb{R}$, and
\begin{equation}\label{level-1-pre-per-eq2}
\liminf\frac{1}{n}\log\sum_{y\in J_n,\frac{S_n(k)(y)}{n}\in
U}p_{n,f}(y)\ge-\inf_{\mu\in\mathcal{M}(J,T),\mu(k)\in
U}\{P(T,f)-\mu(f)-h_{\mu}(T)\}
\end{equation} for all open sets
$U\subset\mathbb{R}$. For each $\mu\in\mathcal{M}(J,T)$ and  each
$\varepsilon$ small enough we have
\begin{equation}\label{level-1-pre-per-eq3}
\lim\frac{1}{n}\log\sum_{y\in U_{n,\mu,\varepsilon}}
p_{n,f}(y)=-\inf_{\mu'\in G_{k,\mu,\varepsilon}\cap\mathcal{M}(J,T)}
\{P(T,f)-\mu'(f)-h_{\mu'}(T)\}
\end{equation}
\[=\lim\frac{1}{n}\log\sum_{y\in \overline{U_{n,\mu,\varepsilon}}}
p_{n,f}(y)=-\inf_{\mu'\in
\overline{G_{k,\mu,\varepsilon}}\cap\mathcal{M}(J,T)}\{P(T,f)-\mu'(f)-h_{\mu'}(T)\},
\]
 with $U_{n,\mu,\varepsilon}=\{y\in
J_n:|\frac{S_n(k)(y)}{n}-\mu(k)|>\varepsilon\}$ and
$G_{k,\mu,\varepsilon}=\{\mu'\in\mathcal{M}(J):|\mu'(k)-\mu(k)|>\varepsilon\}$.

 \subsubsection{Lyapunov exponents}\label{level-1-pre-per-Lya}
 We assume here that $T$ is not conjugate to
$z\mapsto z^{\pm d}$, and  we specialize \S \ref{level-1-pre-per} by
taking
 $f=k_t$ ($t$ any real) and $k=k_{-1}$, so that
    (\ref{level-1-pre-per-eq1})-(\ref{level-1-pre-per-eq3})   hold
    with $\log|{T^n}'(y)|$, $\frac{|{T^n}'(y)|^{-t}}{\sum_{z\in J_n}|{T^n}'(z)|^{-t}}$,
    $\chi(\mu)$ in place of
    $S_n(k)(y)$,     $p_{n,f}(y)$, $\mu(k)$,  respectively. Let us detail the case where  $\mu=\mu_t$ in (\ref{level-1-pre-per-eq3}) (recall that $\mu_t$ is the unique equilibrium state for $k_t$), and first note that the members of (\ref{level-1-pre-per-eq3})
     are strictly negative. Furthermore,
(\ref{level-1-pre-per-eq3}) is specified with  Theorem
\ref{level-1-Lya-theo}  since
\[\sum_{y\in
U_{n,\mu_t,\varepsilon}} p_{n,k_t}(y)=\nu_{n,k_t}(\{\mu\in\mathcal{M}(J):
|\mu({k_{-1}})-\chi(\mu_t)|>\varepsilon\})=\nu_{n,k_t}(G_{k_{-1},\mu_t,\varepsilon}).\]More
precisely,
from the properties of $I^{k_t}_{k_{-1}}$, and since
\begin{equation}\label{level-1-pre-per-Lya-eq8}
P(t)=\lim\frac{1}{n}\log\sum_{z\in J_n}|{T^n}'(z)|^{-t}
\end{equation}
 by (\ref{ex-pre-per-eq4}),
we deduce  the
following relations for each $(t,s)\in\mathbb{R}^2$,
\begin{equation}\label{level-1-pre-per-Lya-eq9}
h_{\mu_{s}}(T)-t\chi(\mu_s) =\left\{
\begin{array}{ll}
\lim\frac{1}{n}\log\sum_{\{y\in
J_n:\frac{\log|{T^n}'(y)|}{n}>\chi(\mu_s)\}}
{|{T^n}'(y)|^{-t}} &
\textnormal{if}\ s\le t
\\ \\
\lim\frac{1}{n}\log\sum_{\{y\in
J_n:\frac{\log|{T^n}'(y)|}{n}<\chi(\mu_s)\}}
{|{T^n}'(y)|^{-t}} &
\textnormal{if}\ s\ge t,
\end{array}
\right.
\end{equation}
 and for each real $t$ and each $\varepsilon>0$ small enough,
\begin{equation}\label{level-1-pre-per-Lya-eq10}
\lim\frac{1}{n}\log\sum_{\{y\in
J_n:|\frac{\log|{T^n}'(y)|}{n}-\chi(\mu_t)|>\varepsilon\}}
\frac{|{T^n}'(y)|^{-t}}{\sum_{z\in J_n}|{T^n}'(z)|^{-t}}=
\end{equation}
\[\max\{h_{\mu_{s_{\chi(\mu_t)+\varepsilon}}}(T)-t\varepsilon,
h_{\mu_{s_{\chi(\mu_t)-\varepsilon}}}(T)+t\varepsilon\}-h_{\mu_t}(T).\]
Taking  $s=t$ in (\ref{level-1-pre-per-Lya-eq9}) gives   formulas for $P(t)$, and taking
$t=0$ in
(\ref{level-1-pre-per-Lya-eq9}) yields
 the following formula for the entropy valid for each real $s$,
\begin{equation}\label{level-1-pre-per-Lya-eq11}
h_{\mu_{s}}(T)=\left\{
\begin{array}{ll}
\lim\frac{1}{n}\log\textnormal{Card}\{y\in
J_n:\frac{\log|{T^n}'(y)|}{n}>\chi(\mu_s)\}
 & \textnormal{if}\ s\le 0
\\ \\
\lim\frac{1}{n}\log\textnormal{Card}\{y\in
J_n:\frac{\log|{T^n}'(y)|}{n}<\chi(\mu_s)\}
 & \textnormal{if}\ s\ge 0.
\end{array}
\right.
\end{equation}
We obtain also from
 the last assertion  of  Theorem \ref{level-1-Lya-theo},
\[
 \forall t\in\mathbb{R},\ \ \ \ \ \lim\sum_{y\in
J_n}\frac{|{T^n}'(y)|^{-t}}{\sum_{z\in
J_n}|{T^n}'(z)|^{-t}}\frac{\log|{T^n}'(y)|}{n}=\chi(\mu_t).
\]
Note that the strict inequality can be replaced by an inequality  in
(\ref{level-1-pre-per-Lya-eq9}), (\ref{level-1-pre-per-Lya-eq10}),
(\ref{level-1-pre-per-Lya-eq11}); applying that to (\ref{level-1-pre-per-Lya-eq9}) with $s=t$,  we recover (\ref{level-1-pre-per-Lya-eq8}).

\subsection{Pre-images (resp. periodic points), $f$ H\"{o}lder
continuous}\label{pre-per}

 Let $f$ be a H\"{o}lder continuous function on $J$,
  let $x\in J$,
  put
$J_n=\{T^{-n}(x)\}$ (resp. $J_n=\{y\in J:T^n(y)=y\}$), and let $\nu_{n,f}$
defined analogously to \S \ref{subsection-separated-set}, namely
 $\nu_{n,f}=\sum_{y\in J_n}
p_{n,f}(y)\delta_{\mu_{y,n}}$ with
$p_{n,f}(y)=\frac{e^{S_n(f)(y)}}{\sum_{z\in J_n}e^{S_n(f)(z)}}$ for
  all $y\in J_n$. Then (\ref{ex-pre-per-eq4}) holds when  $k$ is H\"{o}lder continuous (\cf \cite{prz_BBMS_90} for preimages,
 \cite{parry-poll_PMH_75} for periodic points) and since $f+g$ is
H\"{o}lder continuous when $g$ is,  by taking the limit in
(\ref{ex-pre-per-eq3}) we get $\mathbb{L}(\widehat{g})=Q_f(g)$ for all such $g$,
  so that the hypothesis of
Theorem \ref{LDP} is fulfilled.  Consequently,  all the
conclusions  of \S  \ref{subsection-separated-set-level-2},
\S \ref{level-1-pre-per} and
 \S \ref{level-1-pre-per-Lya}
  hold verbatim with $J_n$ as above. In
particular,  since $f$ has a unique equilibrium state $\mu_f$,
(\ref{ex-pre-per-eq10}) holds, and the members of
(\ref{level-1-pre-per-eq3}) with
$\mu=\mu_f$ are strictly negative.

\begin{remark}\label{Birk-general}
The results of \cite{lop} concerning the Birkhoff averages with respect to  the measure of maximal entropy $\mu_0$,  can be easily recovered from Theorem \ref{LDP} and Theorem \ref{level-1-Lya-theo}.
Indeed,  take $f=0$, $J_n=J$ and  $p_{n,0}=\mu_0$ for all $n\ge 1$, so that
\[\nu_{n,0}(\cdot)=\mu_0(\{y\in J:\mu_{y,n}\in\cdot\}),\]
and
\begin{equation}\label{level-2-birk-eq2}
\lim\frac{1}{n}\log\nu_{n,0}(e^{n\widehat{g}})=
\lim\frac{1}{n}\log\mu_0(e^{S_n(g)})=Q_0(g)=P(g)-\log d
\end{equation}
 for all $g\in C(J)$  (\cite{lop}, Remark 1), hence
the hypotheses of Theorem \ref{LDP} is satisfied, and we  recover the level-2
 large deviation principle for $(\nu_{n,0})$ as in  Theorem 7  of \cite{lop}. In particular,
 the expression of the entropy is given for each
$\mu\in\mathcal{M}(J,T)$ (and any distance $\rho$ as in
(\ref{ex-pre-per-eq8.1})) by
\begin{equation}\label{level-2-birk-general-eq7}
h_{\mu}(T)=\log
d+\lim_{\varepsilon\rightarrow 0}\lim\frac{1}{n}\log\mu_0(\{y\in
J:\rho(\mu,\mu_{y,n})<\varepsilon\})
\end{equation}
(we note that  in Remark 5 of \cite{lop} the term $\log d$ is missing).
 Assuming  that $T$ is not
conjugate to $z\mapsto z^{\pm d}$, and taking $f=0$ and
$k=k_{-1}$ in Theorem \ref{level-1-Lya-theo},  the analogues  of (\ref{level-1-pre-per-Lya-eq9}) and
(\ref{level-1-pre-per-Lya-eq10}) are respectively for each real $s$,
\begin{equation}\label{level-1-Birk-general-eq1}
 h_{\mu_{s}}(T)=\left\{
\begin{array}{ll}
\log d+\lim\frac{1}{n}\log\mu_0(\{y\in
J:\frac{\log|{T^n}'(y)|}{n}>\chi(\mu_s)\}) & \textnormal{if}\ s\le 0
\\ \\
\log d+\lim\frac{1}{n}\log\mu_0(\{y\in
J:\frac{\log|{T^n}'(y)|}{n}<\chi(\mu_s)\}) & \textnormal{if}\ s\ge
0,
\end{array}
\right.
\end{equation}
and for each $\varepsilon>0$ small enough,
\begin{equation}\label{level-1-Birk-general-eq2}
\lim\frac{1}{n}\log\mu_0(\{y\in
J:|\frac{\log|{T^n}'(y)|}{n}-\chi(\mu_0)|>\varepsilon\})=
\end{equation}
\[
\max\{h_{\mu_{s_{\chi(\mu_0)+\varepsilon}}}(T),
h_{\mu_{s_{\chi(\mu_0)-\varepsilon}}}(T)\}-\log d.\]
The formula (\ref{level-1-Birk-general-eq1}) can easily be deduced from \cite{lop}, and (\ref{level-1-Birk-general-eq2}) corresponds to Corollary 2 of \cite{lop}.
\end{remark}

\section{Generalization - Examples}\label{section-extension}

It is easy to see that  the proof of  Theorem  \ref{LDP} does not depend on the dynamics of rational maps. In fact, it rests on two basic ingredients:   the equality of functionals $\mathbb{L}(\widehat{\cdot})=Q_f(\cdot)$,  and the approximation property (given by Theorem \ref{Lopes-theorem}) combined  with Lemma \ref{relation-l*-Hmu}. These conditions
involving nets of measures on the phase space and notions of thermodynamical formalism,
 they  can be defined (and Lemma \ref{relation-l*-Hmu}  proved) as well for general dynamical systems in the sense of Ruelle's book (\cite{Rue-78}).  In this section, after stating the  general version (Theorem \ref{LDP-general}), we observe that the main  result of \cite{Gurevich-Tempelman-05-PTRF} concerning  the multidimensional  full shift amounts to  the approximation property, and thus furnishes  an example of  distinct nature from the one dimensional system given by  rational  maps, but sharing similar large deviation principles  for the same kinds of measures  with moreover the same proof (Theorem \ref{LDP-fullshift-atomic}). Again here   the techniques of  \cite{Kifer-90-TAMS} do not apply   when $f$ has several equilibrium states.

Let $\Omega$ be a non-empty compact metrizable space, let $l$ be a strictly positive integer,
put  $\mathbb{Z}^{l}_+=\{x\in\mathbb{Z}^l: x_i\ge 0, 1\le i\le l\}$, and let $\tau$ be a
 representation  of the semi-group $\mathbb{Z}^{l}_+$ (resp. group $\mathbb{Z}^{l}$) in the semi-group of continuous endomorphisms
 (resp. group of homeomorphisms) of  $\Omega$.
Let  $C(\Omega)$, $\mathcal{M}(\Omega)$,
$\mathcal{M}^{\tau}(\Omega)$, $\mathcal{M}^{\tau}_f(\Omega)$,
$h^{\tau}_{\cdot}$, $P^{\tau}(\cdot)$ be  the obvious analogues of
$C(J)$, $\mathcal{M}(J)$, $\mathcal{M}(J,T)$, $\mathcal{M}_f(J,T)$,
$h_{\cdot}(T)$, $P(T,\cdot)$ defined in  \S \ref{general-results}, and  assume that
 $h^{\tau}$ is finite and upper semi-continuous. For each $f\in C(\Omega)$ we define the function
 $I^f$  on
$\mathcal{M}(\Omega)$ by
\begin{displaymath}\label{LDP-upper-bounds-eq1}
I^f(\mu)=\left\{
\begin{array}{ll}
P^{\tau}(f)-\mu(f)-h^{\tau}(\mu) & \textnormal{if}\
\mu\in\mathcal{M}^{\tau}(\Omega)
\\ \\
+\infty & \textnormal{if}\
\mu\in{\mathcal{M}}(\Omega)\verb'\'\mathcal{M}^{\tau}(\Omega),
\end{array}
\right.
\end{displaymath}
so that  $I^f$ vanishes exactly on $\mathcal{M}^{\tau}_f(\Omega)$.  The analogue of the approximation property of Theorem \ref{Lopes-theorem}
 takes  the following general form.

\begin{property}\label{general-approx-property}
For each $\mu\in\mathcal{M}^{\tau}(\Omega)$  there is a net  $(k_i)_{i\in\wp_\mu}$ in $C(\Omega)$ such that  $k_i$ has a unique equilibrium state $\mu_i$ for all $i\in\wp_\mu$, and  the net $(\mu_i)_{i\in\wp_\mu}$ satisfies $\lim\mu_i=\mu$ and
$\lim h^{\tau}_{\mu_i}=h^{\tau}_{\mu}$.
\end{property}

We can now state the general version of Theorem  \ref{LDP},  whose proof is entirely  similar (just  take account of Remark \ref{remark-upper-bounds} for a), and use Property \ref{general-approx-property} in place of Theorem \ref{Lopes-theorem} for b)). We let the reader establish  the analogue of Corollary \ref{LDP-level-1}.

\begin{theorem}\label{LDP-general}
Let $f\in C(\Omega)$, let $(\nu_\alpha)$ be a net of Borel probability measures on $\mathcal{M}(\Omega)$, let $(t_\alpha)$ be a net in $]0,+\infty[$ converging to $0$,  let $\overline{\mathbb{L}}$ be the associated large deviation functional, and assume there is a dense set $C\subset C(\Omega)$ such that
 \begin{equation}\label{LDP-general-eq2}
\forall g\in C,\ \ \ \ \ \ \  \overline{\mathbb{L}}(\widehat{g})\le
P^{\tau}(f+g)-P^{\tau}(f).
\end{equation}
\begin{nitemize}
\item[\textnormal{a)}] For each closed set $F\subset\mathcal{M}(\Omega)$ we have
\[\limsup\  t_\alpha\log\nu_\alpha(F)\le-\inf_{\mu\in F} I^f(\mu).\]
   Each limit point of $(\nu_\alpha)$ has its support included in $\mathcal{M}^{\tau}_f(\Omega)$; in particular,
$\lim\nu_\alpha=\delta_{\mu_f}$ when $f$ has a unique equilibrium sate $\mu_f$.
\item[\textnormal{b)}]
 If Property \ref{general-approx-property}  holds and  $\mathbb{L}(\widehat{g})$ exists for all $g\in C$ with an equality in  (\ref{general-approx-property}), then for each open  set $G\subset\mathcal{M}(\Omega)$ we have
 \[\liminf\  t_\alpha\log\nu_\alpha(G)\ge-\inf_{\mu\in G} I^f(\mu);\]
when $G$ is moreover convex and contains some invariant measure,  we have
\begin{equation}\label{LDP-general-eq4}
\lim
t_{\alpha}\log\nu_\alpha(G)=\lim
t_{\alpha}\log\nu_\alpha(\overline{G})=-\inf_{\mu\in
\overline{G}}I^f(\mu)=-\inf_{\mu\in G\cap\cE'}I^f(\mu),
\end{equation}
where $\cE'$ is the set of  equilibrium states of all elements in $\{k_i:i\in\wp_\mu,\mu\in\mathcal{M}^\tau(\Omega)\}$.
 In particular we obtain    for  each $\mu\in\mathcal{M}^{\tau}(\Omega)$ and each  convex
local basis $\mathcal{G}_\mu$  at $\mu$,
\[
h^{\tau}(\mu)=P^{\tau}(f)-\mu(f)+\inf\{\lim\ t_\alpha\log\nu_\alpha(G):
G\in\mathcal{G}_\mu\}.
\]
\end{nitemize}
\end{theorem}

\begin{remark}\label{remark-strengthened}
Property \ref{general-approx-property} permits to get more than just
a large deviation principle  with rate function $I^f$. Indeed, although the two first equalities in  (\ref{LDP-general-eq4})  are still true without Property \ref{general-approx-property} (combining (\ref{preliminaries-eq2}) with the convexity of $G$ and  the fact that  $I^f$ is affine), the last equality  is given by Property \ref{general-approx-property}, and
shows that   the  exponential behavior of  $\nu_\alpha(G)$  is controlled by the entropy of measures which are unique equilibrium states. More precisely, for each $\varepsilon>0$ there exists $k_\varepsilon\in C(\Omega)$ with unique equilibrium state $\mu_{k_\varepsilon}\in G$ such that eventually,
\[|t_{\alpha}\log\nu_\alpha(G)+I^f(\mu_{k_\varepsilon})|<\varepsilon.\]
\end{remark}

\begin{remark}
Condition (\ref{LDP-general-eq2}) corresponds  to Assumptions (a), (b) of \cite{Pol_DCDS_96} for  suitable measures and $f=0$, and consequently Theorem 2 (resp. Corollary 2.1) of \cite{Pol_DCDS_96} follows immediately from the first (resp. second) assertion of Theorem \ref{LDP-general} a).
\end{remark}

\begin{remark}\label{remark-upper-bounds-gen}
Theorem \ref{LDP-general} a) can be applied to the same dynamics  as the one considered in \S \ref{general-results} with  $T$ any  rational map of degree $d\ge
2$ (without  hyperbolicity condition). For instance, let us  consider the sequence of measures obtained with  preimages and $f$ H\"{o}lder continuous as in \S \ref{pre-per}. It is known that
(\ref{ex-pre-per-eq4}) still holds when $k$ is H\"{o}lder continuous and  $P(T,k)>\sup_J k$, and for general $k\in C(J)$ the equality in (\ref{ex-pre-per-eq4}) has to be replaced by $``\le"$ and the limit
by  a upper limit (\cite{prz_BBMS_90}). Consequently,  when
$P(T,f)>\sup_J f$,  (\ref{LDP-general-eq2}) is fulfilled  with $C$ the set of H\"{o}lder continuous
 functions
  by taking the upper limit in
(\ref{ex-pre-per-eq3}). As a conclusion, we recover  the large deviation results
of \cite{pol_CMP_96}.
\end{remark}

\subsection{The multidimensional full shift}\label{multidimensional-fullshift}

Let $S$ be a finite set, let $\delta\in]0,1[$,  put
$\Omega=S^{\mathbb{Z}^l}$ endowed with the metric $\rho(\xi,\eta)=\delta^{\min\{\max_{1\le i\le l} |x_i|:x\in\Z^l,\xi_x\neq\eta_x\}}$, and let $\tau$ be the
action of $\mathbb{Z}^l$ on $\Omega$ by translations, \ie
 $(\tau^y \xi)_x=\xi_{x+y}$ for all $x,y$ in $\Z^l$  and $\xi\in \Omega$.
A recent result of Gurevich and Tempelman (\cite{Gurevich-Tempelman-05-PTRF}, Theorem 1) can be formulated  in the following way (in fact, the authors show that the continuous functions  appearing  in Property \ref{general-approx-property} can be obtained as mean energy functions associated to some  summable interactions).

\begin{theorem}\label{approx-prop-finite-spin-general-f}
\textnormal{\textbf{(Gurevich-Tempelman)}}
The multidimensional full shift fulfils  Property \ref{general-approx-property}.
\end{theorem}

We shall consider nets of measures similar to those of \S \ref{subsection-separated-set-level-2} and \S \ref{pre-per}.  The
``time" $n$ is  replaced by a multidimensional one, namely a net
$(\Lambda_\alpha)_{\alpha\in\wp}$  of
finite subsets of $\mathbb{Z}^l$.
Recall that  $(\Lambda_\alpha)$ is said to converge to $\infty$ in the sense of van Hove (denoted $\Lambda_\alpha\nearrow\infty$)  when $\lim|\Lambda_\alpha|=+\infty$ (where
$|\Lambda_\alpha|=\textnormal{Card\ }\Lambda_\alpha$)
and $\lim|(\Lambda_\alpha+x)\verb´\´\Lambda_\alpha|/|\Lambda_\alpha|=0$ for all $x\in\mathbb{Z}^l$. Put $\mathbb{Z}^l_>=\{x\in\mathbb{Z}^l: x_i>0, 1\le i\le l\}$,  and  for each $x\in\mathbb{Z}^l_>$ define
 $\Lambda(x)=\{y\in\mathbb{Z}^l:0\le y_i<x_i, 1\le i\le l\}$ and  the set of $x$-periodic configurations
  $\textnormal{Per}_x=\{\xi\in\Omega:\tau^{y}\xi=\xi\textnormal{\ for\
all\ }y\in\mathbb{Z}^l(x)\}$, where
$\mathbb{Z}^l(x)$ is the subgroup of $\mathbb{Z}^l$ generated by
$(x_1,0,...,0)$,...,$(0,...,0,x_l)$. Note that  $\Lambda(x)\nearrow\infty$ when $\mathbb{Z}^l_>$ is directed by the lexicographic order and $\lim x_i=+\infty$ for all $i\in\{1,...,l\}$, which will be assumed in what follows.
For each $\xi\in\Omega$   we define
\[\mu_{\xi,\alpha}=\frac{1}{|\Lambda_\alpha|}
\sum_{x\in\Lambda_\alpha}\delta_{\tau^x \xi}\]
and
\[W_\alpha(\xi)=\mu_{\xi,\alpha}.\]
In place of  the sequences
$(J_n)$ we shall consider some nets $(\Omega_\alpha)$ of finite
subsets of $\Omega$;
 ${\mu_{y,n}}$ (resp.  $W_n$)
 is replaced by  $\mu_{\xi,\alpha}$ (resp. $W_\alpha$),
and  the probability measures $p_{n,f}$ by  $p_{\alpha,f}$ with
  \[p_{\alpha,f}(\xi)=\frac{e^{\sum_{x\in\Lambda_\alpha}f(\tau^x \xi)}}{\sum_{\xi'\in
\Omega_\alpha}e^{\sum_{x\in\Lambda_\alpha}f(\tau^x \xi')}}.\]
  We shall  obtain large deviations for  nets $(\nu_{\alpha,f})$ defined by
\[\nu_{\alpha,f}=W_\alpha[p_{\alpha,f}]=\sum_{\xi\in
\Omega_\alpha}p_{\alpha,f}(\xi)\delta_{\mu_{\xi,\alpha}},\] each of which is associated with  some net $\Lambda_\alpha\nearrow\infty$ and corresponds to some
way to obtain the pressure, in the sense that
\begin{equation}\label{section-extension-eq1}
\forall g\in C(\Omega),\ \ \ \ \ \ \   \mathbb{L}(\widehat{g})=\lim\frac{1}{|\Lambda_\alpha|}\log\nu_{\alpha,f}
(e^{|\Lambda_\alpha|\widehat{g}})=
P^{\tau}(f+g)-P^{\tau}(f).
\end{equation}
Once proved the above equality,
the conclusion follows from  Theorem \ref{LDP-general}  and  Theorem \ref{approx-prop-finite-spin-general-f}.

  The next result establishes  the  large deviation principle  for finite supported measures constructed from maximal separated sets (resp. periodic configurations)  as for rational maps.   The explicit forms of the large deviations as well as the  formulas for entropy  are entirely analogue to those for rational maps,  after obvious changes of notations (\cf \S \ref{subsection-separated-set-level-2}, \S \ref{pre-per}). Both cases are    new, and the one of periodic configurations  generalizes  the version of   \cite{Eiz-Kif-Weis-94-CMP} and \cite{Olla-88-PTRF} proved  for $f=0$ (\cf Remark \ref{remark-known-results-full-shift}).

\begin{theorem}\label{LDP-fullshift-atomic}
Let $\Lambda_\alpha\nearrow\infty$ and $(\Omega_\alpha)$ given in one of the following ways.
 \begin{itemize}
 \item[(a)] $\Omega_\alpha$ is
 a maximal $(\varepsilon,\Lambda_\alpha)$-separated set for some $\varepsilon<\delta$;
\item[(b)]
 $\Lambda_\alpha=\Lambda(\alpha)$ and
      $\Omega_\alpha=\textnormal{Per}_\alpha$ for all $\alpha\in\mathbb{Z}^l_>$.
\end{itemize}
Then for each $f\in C(\Omega)$,  all the conclusions of Theorem \ref{LDP-general} hold with  $t_\alpha=|\Lambda_\alpha|^{-1}$ and
\[\nu_{\alpha,f}=\sum_{\xi\in
\Omega_\alpha}p_{\alpha,f}(\xi)\delta_{\mu_{\xi,\alpha}}.\]
\end{theorem}

\begin{proof}
 Since
\[\forall g\in C(\Omega),\ \ \ \ \ \ \   \nu_{\alpha,f}
(e^{|\Lambda_\alpha|\widehat{g}})=\frac{\sum_{\xi\in
\Omega_\alpha}e^{\sum_{x\in\Lambda_\alpha}(f+g)(\tau^x \xi)}}{\sum_{\xi'\in
\Omega_\alpha}e^{\sum_{x\in\Lambda_\alpha}f(\tau^x \xi')}},\]  in both cases (\ref{section-extension-eq1})  follows from Theorem 2.2 of \cite{Ruelle-73-TAMS}.
\end{proof}

\begin{remark}\label{remark-known-results-full-shift}
The results of Theorem \ref{LDP-fullshift-atomic} (b)   are  similar to those of \cite{Olla-88-PTRF}, where  nets of the form $(W_\alpha[\mu^{\Z^l}])$ are studied, for some fixed probability measure  $\mu$ on $S$. More precisely,
 when the spin  space in \cite{Olla-88-PTRF} is finite and $\mu$ is the uniform distribution,
Theorem 3.5 (resp. Theorem 4.2) of that paper is exactly  the large deviation upper (resp. lower) bounds of Theorem \ref{LDP-fullshift-atomic} (b) with $f=0$.
 Therefore, Theorem \ref{LDP-fullshift-atomic}  (b) generalizes this particular case  allowing  any $f\in C(\Omega)$; it also
  extends in the same way the full shift case in Theorem C of
 \cite{Eiz-Kif-Weis-94-CMP}  where only $f=0$ is considered (on the other hand, Theorem C of
 \cite{Eiz-Kif-Weis-94-CMP} is much more general since it  holds for any  subshift  of finite type satisfying strong specification).
 \end{remark}

  \begin{remark}\label{remark-known-results-full-shift-conv-measures}
  The measures $p_{\alpha,f}$  as in Theorem \ref{LDP-fullshift-atomic}  (b)   have been considered in \cite{Ruelle-73-TAMS} for any $f\in C(\Omega)$; in  particular,  Theorem 3.2 of that paper   establishes  that every limit point of $(p_{\alpha,f})$ belongs to $\mathcal{M}^{\tau}_{f}(\Omega)$. It is easy to see that this result  can be recovered from Theorem \ref{LDP-fullshift-atomic}. Indeed, let $\mu$ be a limit point of $(p_{\alpha,f})$, or equivalently a limit point of $(\sum_{\xi\in
\textnormal{Per}_\alpha}p_{\alpha,f}(\xi){\mu_{\xi,\alpha}})$. Then $\mu$ is the barycenter of some  limit point $\nu$ of
$(\nu_{\alpha,f})$ (recall that the barycenter map is  weak$^*$ continuous). Since  $\nu$ is supported by $\mathcal{M}^{\tau}_{f}(\Omega)$ (Theorem \ref{LDP-general} a)), we conclude that $\mu$ belongs to $\mathcal{M}^{\tau}_{f}(\Omega)$.
\end{remark}

\bigskip

\noindent\textbf{Acknowledgments.}
  The author wishes to thank Juan  Rivera-Letelier for  many helpful discussions on  the dynamics of rational maps. Thanks are also due  for
the financial   support and the warm hospitality  during  various visits at the Universidad Cat\'{o}lica del Norte. This work has been supported by FONDECYT grant No. 1070045.


\begin{thebibliography}{100}

\bibitem{bal}
P. Baldi. Large deviations and stochastic homogenization.
\textit{Ann. Mat. Pura Appl. 151 (1988), 161-177}.

\bibitem{com3}
H. Comman. Variational form of the large deviation functional.
\textit{Statistics and Probability Letters 77 (2007), no. 9 ,
931-936.}




\bibitem{com-TAMS-03}
H. Comman. Criteria for large deviations. \textit{Trans. Amer.
Math. Soc. 355 (2003), no. 7,  2905-2923.}




\bibitem{dem}
A. Dembo, O. Zeitouni. \textit{Large deviations techniques and
applications}, Second Edition, Springer, New-York, 1998.


\bibitem{Eiz-Kif-Weis-94-CMP}
A. Eizenberg, Y. Kifer, B. Weiss. Large deviations for
$\mathbb{Z}^d$-actions. \textit{Comm. Math. Phys 164 (1994),
433-454}.




\bibitem{Gurevich-Tempelman-05-PTRF}
B. M. Gurevich,  A. A. Tempelman. Markov approximation of
homogeneous lattice random fields. \textit{Probab. Theory Related
Fields 131 (2005), 519-527}.


\bibitem{Keller98}
G. Keller. \textit{Equilibrium states in ergodic theory}, London Mathematical Society Student Texts 42, Cambridge University Press, Cambridge, 1998.


\bibitem{Kel-Now_CMP_92}
G. Keller,  T.  Nowicki. Spectral theory, Zeta functions and
the distribution of periodic points for Collet-Eckmann maps.
\textit{Comm. Math. Phys. 149 (1992), 31-69}.




\bibitem{Kifer-94-CMP}
Y. Kifer. Large deviations, averaging and periodic orbits  of dynamical systems. \textit{Comm. Math. Phys. 162 (1994), no. 1, 33-46}.




\bibitem{Kifer-90-TAMS}
Y. Kifer. Large deviations in dynamical systems and stochastic
processes. \textit{Trans. Amer. Math. Soc. 321 (1990), no. 2,
505-524}.






\bibitem{lop}
A. O. Lopes. Entropy and large deviation. \textit{Nonlinearity 3
(1990), 527-546.}



\bibitem{Melbourne_Nicol-08-TAMS360}
I. Melbourne, M. Nicol. Large deviations for nonuniformly hyperbolic systems. \textit{Trans. Amer. Math. Soc. 360 (2008), 6661-6676}.




\bibitem{Olla-88-PTRF}
S. Olla. Large deviations for Gibbs random fields. \textit{Probab.
Theory and Related Fields 77 (1988), 343-357}.


\bibitem{parry-poll_PMH_75}
W. Parry, M. Pollicott.  \textit{Zeta functions and the periodic orbit
structure of hyperbolic dynamics}, Asterisque 187-188 (1990) ISSN: 0303-1179.




\bibitem{Pollicott-Sridharan-07-JDSGT}
M. Pollicott. Large deviations results for periodic points of a
rational map. \textit{J. Dyn Syst. Geom. Theor. 5 (2007), no. 1, 69-67}.


\bibitem{pol_CMP_96}
M. Pollicott, R. Sharp. Large deviations and the distribution of
preimages of rational maps. \textit{Comm. Math. Phys. 181 (1996),
733-739}.

\bibitem{Pol_DCDS_96}
M. Pollicott. Closed geodesic distribution for manifolds of non-positive curvature, \textit{Discrete and Continuous Dynamical Systems 2 (1996), no. 2, 153-161}.


\bibitem{prz_BBMS_90}
F. Przytycki.  On the Perron-Frobenius-Ruelle operator for rational
maps on the Riemann sphere and for H\"{o}lder continuous functions.
\textit{Bol. Bras. Mat. Soc. 20 (1990), 95-125}.

\bibitem{Rey-Bellet_Young-08-ETDS28}
L. Rey-Bellet, L. S.  Young. Large deviations in non-uniformly hyperbolic
dynamical systems. \textit{Ergod. Th.  Dynam. Sys. 28 (2008),  587-612}.


\bibitem{roc}
R. T. Rockafeller. \textit{Convex analysis}, Princeton University
Press, Princeton, 1970.


\bibitem{Rue-82-ETDS2}
D. Ruelle. Repellers for analytic maps. \textit{Ergod. Th.  Dynam. Sys. 2 (1982),  99–107}.



\bibitem{Rue-78}
D. Ruelle. \textit{Thermodynamic formalism},  Addison-Wesley,
Reading, MA, 1978.




\bibitem{Ruelle-73-TAMS}
D. Ruelle. Statistical mechanics on a compact set with
$\mathbb{Z}^d$-action satisfying expansiveness and specification.
\textit{Trans. Amer. Math. Soc. 185 (1973), 237-251}.

\bibitem{sch}
H. H. Schaefer. \textit{Topological vector spaces}, Springer-Verlag,
New York, 1964.



\end{thebibliography}
\end{document}